\documentclass[letterpaper,12pt]{article}
\usepackage{cite}
\usepackage{tabularx} % extra features for tabular environment
\usepackage{amsmath,amsthm,amssymb}  % improve math presentation
\usepackage{amsmath}
\usepackage{amsfonts}
\usepackage{amssymb}
\usepackage{graphicx} % takes care of graphic including machinery
\usepackage[margin=0.8in,letterpaper]{geometry} % decreases margins
\usepackage[english]{babel}
\usepackage{caption}
\usepackage{float}
\usepackage[table]{xcolor}
\usepackage{array}
\usepackage{algorithm}
\usepackage{algpseudocode}
\usepackage{cases}
\usepackage{titlesec}
\usepackage{tikz}
\usepackage[utf8]{inputenc}
\usepackage[T1]{fontenc}
\usepackage{bm}
\usepackage{mathtools}
\usepackage{extarrows}
\usepackage{booktabs} 
\usepackage{dsfont}
\usepackage{authblk}
\usepackage{multirow}
\usepackage{hyperref}
\usepackage{parskip}

\usepackage{csquotes}

\newtheorem{theorem}{Theorem}
\newtheorem{lemma}{Lemma}
\newtheorem{proposition}{Proposition}
\newtheorem{corollary}{Corollary}
\newtheorem{definition}{Definition}
\newtheorem{example}{Example}

\newtheorem{remark}{Remark}

\usepackage{microtype}      % microtypography
\usepackage{lipsum}
\usepackage{nicefrac}
\usepackage{fancyhdr}       % header
\usepackage{orcidlink}

\author{El-Mehdi Mehiri\orcidlink{0000-0002-7164-3658}\\
Mines Saint-Étienne, CMP, Department of Manufacturing Sciences and Logistics, 
F-13120 Gardanne, France.\\ \texttt{elmehdi.mehiri@emse.fr}\\
\texttt{mehiri314@gmail.com}}

\title{\textbf{Block-Separated Overpartitions: Fibonacci Structure and Euler Factorization}}

\date{}
\begin{document}
\maketitle

\begin{abstract}
\noindent We introduce and study \emph{block-separated overpartitions}, a constrained family of overpartitions in which no two consecutive distinct part-blocks are both overlined. This local restriction produces a new sequence that naturally interpolates between classical partitions and unrestricted overpartitions. We show that the internal decoration of distinct part-blocks is governed by Fibonacci-type combinatorics: once the set of distinct part-sizes is fixed, the admissible overlining patterns are counted by Fibonacci numbers. This leads to a symmetric-function expansion of the generating function and a two-state transfer-matrix formulation. After extracting the Euler product, we obtain normalized recurrences, second-order scalar recurrences, determinantal representations, and a continued-fraction description of finite truncations. Finally, we determine the asymptotic growth of the counting function, and prove that block-separated overpartitions share the same exponential scale as ordinary partitions, with a modified subexponential constant.
\end{abstract}

\textbf{Keywords:} Overpartitions;  partitions; Fibonacci numbers; transfer matrix; Euler product; continued fractions; asymptotics; generating function.
 
\textbf{MSC (2020):} 05A17; 11P81; 05A15; 11P82; 11P84.

% \begin{abstract}
%   We introduce and study a new restricted family of overpartitions, called
% block-separated overpartitions, in which no two consecutive distinct 
% part-size blocks may both be overlined.  
% Using a two-state transfer-matrix automaton, we derive a closed matrix-product 
% expression for the ordinary generating function, establish an Euler-type 
% factorization, and obtain an explicit normalized recurrence suitable for 
% computation of arbitrary coefficients.  
% We further prove that the possible overlining patterns on the distinct blocks 
% are counted by Fibonacci numbers, giving natural bijections with independent 
% sets on paths, pattern-avoiding binary words, and Fibonacci tilings.  

\section{Introduction}\label{sec:Introduction}

Overpartitions form a fundamental refinement of classical integer partitions:  
one may optionally overline the first occurrence of each distinct part.  
Originally studied in depth by Corteel and Lovejoy~\cite{corteel2004overpartitions}, 
overpartitions have since occupied an important place in   combinatorics.  For instance, a wide range of arithmetic properties of overpartitions --- as well as of overpartitions restricted to odd parts --- have been established by numerous authors. For further details and references, see~\cite{CHEN201562, chen2016ramanujan, chern2018some, hirschhorn2006arithmetic, kim2009short, nadji2025arithmetic, Ahmia202565}.

The overpartition function, denoted by $\overline{p}(n)$, is the number of overpartitions of
n, with the convention that $p(0) = 1$. For example, $\overline{p}(3) = 8$ because there are 8 overpartitions of 3, namely, 

\[
3,  \overline{3},  2 + 1,  \overline{2} + 1,  2 + \overline{1}, \overline{2} + \overline{1}, 1 + 1 + 1,  \text{and } \overline{1} + 1 + 1.
\]

Previous work has considered many structural restrictions on overpartitions and their
combinatorial statistics.  
Several authors studied difference conditions and gap restrictions on
overpartitions,  
others investigated congruence--restricted or arithmetic–progression
refinements,  
while further developments focused on bounded part–size or bounded weight
variants, 
as well as statistical refinements involving ranks, cranks, and related $q$-series.  
These restricted families often give rise to new generating functions with rich analytic
structure, including product expansions, mock–modular behavior, and recursive $q$-difference
equations.  
For detailed results and numerous examples of such restricted overpartition models, we refer
the reader to  
\cite{
Al-Saedi2019251,
Baruah2025,
Bringmann2015,
Chern2019,
Hanson2023,
Hirschhorn2020167,
Isnaini20191207,
Lin202059,
Naika20191,
Sills2010253,
Srivastava2020488,
Tang2023659}.

In this work we propose a constrained family, called \emph{block-separated overpartitions}, which forbids consecutive distinct part-blocks from both being overlined.
This mild local restriction creates a new combinatorial sequence that interpolates between classical partitions and unrestricted overpartitions, while preserving an elegant product-structure and a Fibonacci-type internal combinatorics.

The paper is organized as follows. In Section~\ref{sec:definiton} we introduce block-separated overpartitions and provide initial examples and numerical data. 
Section~\ref{sec:fibonacci} develops the Fibonacci-type combinatorics governing admissible overlining patterns. 
Section~\ref{sec:transfer_matrix} presents a two-state automaton and the associated transfer-matrix formulation. 
In Section~\ref{sec:generating_function} we derive the global generating function. 
Section~\ref{sec:Euler} establishes an Euler-type factorization and normalized recurrences. 
Sections~\ref{sec:Recurrence}–\ref{sec:ContinuerFraction} develop structural consequences, including second-order recurrences, determinantal representations, and a continued-fraction formulation. 
Section~\ref{sec:Asymptotics} determines the asymptotic growth of the number of block-separated overpartitions. 
Finally, Section~\ref{sec:conclusion} discusses conclusions and further research directions.

\section{Definitions and   Examples}\label{sec:definiton}

\begin{definition}
A \emph{partition} of a positive integer $n$ is a finite sequence $\lambda = (d_1^{m_1}, d_2^{m_2}, \dots, d_r^{m_r})$, where $d_1 > d_2 > \dots > d_r \ge 1$,  $m_i \in \mathbb{N}=\{1,2,\ldots\}$, and  
\[
\sum_{i=1}^r m_i d_i = n.
\]
Here $m_i$ denotes the \emph{multiplicity} of the part $d_i$.

An \emph{overpartition} of \(n\) is obtained by optionally overlining the first occurrence of each distinct part.
Denote by \(x_i \in \{0,1\}\) the indicator such that \(x_i = 1\) if and only if the first occurrence in the block \(d_i^{m_i}\) is overlined.

The overpartition \(\lambda\) is said to be a \emph{block-separated overpartition} if no two consecutive distinct blocks are both overlined, that is,
\[
x_i x_{i+1} = 0 \quad \text{for all } i=1,\ldots,r-1.
\]
We denote by \( \mathcal{B}(n) \) the set of all block-separated overpartitions of \( n \), and by
\[
\mathcal{F}(q) = \sum_{n\ge0} b(n) q^n,
\]
its ordinary generating function, where $b(n) = |\mathcal{B}(n)|$.
\end{definition}

\begin{example}
For \( n = 5 \), the $24$ classical overpartitions reduce to \(19\) block-separated ones.
For instance, \(2\!+\!2\!+\!\overline{1}\) is valid, but \(\overline{2}\!+\!2\!+\!\overline{1}\) is forbidden because both distinct blocks (\(2\) and \(1\)) are overlined.
Hence \(b(5)=19\).  The   $19$ admissible overpartitions for $n=5$ are:
\[
\begin{aligned}
&5,\quad \overline{5},\quad\  4+1,\quad \overline{4}+1,\quad 4+\overline{1},\quad 3+2,\quad \overline{3}+2,\quad 3+\overline{2},\quad 3+1+1,\quad \overline{3}+1+1,\quad 3+\overline{1}+1,\\[4pt]
&2+2+1,\quad \overline{2}+2+1,\quad 2+2+\overline{1},\quad 2+1+1+1,\quad \overline{2}+1+1+1,\quad 2+\overline{1}+1+1,\\[4pt]
&1+1+1+1+1,\quad \overline{1}+1+1+1+1.
\end{aligned}
\]

\end{example}

\begin{definition}
For each positive integer \( j \), let
\begin{equation}
    S_j(q) = \sum_{m\ge1} q^{jm} = \frac{q^j}{1-q^j}
\end{equation}
be the generating function for a block consisting of parts of size \( j \) appearing with multiplicity at least \(1\).
\end{definition}

We recall the \emph{infinite $q$-Pochhammer symbol}
\begin{equation}\label{eq:qpoch}
(q)_\infty
  \coloneqq
  \prod_{j=1}^{\infty} (1 - q^j).
\end{equation}

Before presenting the analytic properties of the generating functions, it is
instructive to compare the initial values of the three fundamental sequences
appearing in this work: the number of classical partitions $p(n)$, the number
of overpartitions $\overline{p}(n)$, and the number of block-separated
overpartitions $b(n)$.  
Table~\ref{tab:comparison-p-over-b-rotated} lists these values for
$0 \le n \le 10$.  
As expected, we have
\[
p(n) \le b(n) \le \overline{p}(n),
\]
since block-separated overpartitions interpolate between ordinary partitions
(no overlining) and unrestricted overpartitions (free overlining).  

\begin{table}[H]

\begin{tabular*}{\linewidth}{@{\extracolsep{\fill}} l|lllllllllll@{}}
\toprule
$n$ & 0 & $1$ & $2$ & $3$ & $4$ & $5$ & $6$ & $7$ & $8$ & $9$ & $10$ \\
\midrule
$p(n)$ 
& 1  & 1  & 2  & 3  & 5  & 7  & 11 & 15 & 22 & 30 & 42 \\[2mm]
$b(n)$ 
& 1  & 2  & 4  & 7  & 12 & 19 & 31 & 47 & 72 & 107 & 157 \\[2mm]
$\overline p(n)$ 
& 1  & 2  & 4  & 8  & 14 & 24 & 40 & 64 & 100 & 154 & 232 \\
\bottomrule
\end{tabular*}
\caption{Comparison of the number of partitions $p(n)$, overpartitions $\overline p(n)$,
and block-separated overpartitions $b(n)$ for $0 \le n \le 10$.}
\label{tab:comparison-p-over-b-rotated}
\end{table}

We recall the definition of the Fibonacci numbers. 
The Fibonacci sequence $\{F_n\}_{n\ge 0}$ is defined by the recurrence 
\begin{equation}\label{eq:Fibonacci}
    F_0 = 0, \ F_1 = 1,\quad \forall n\ge 2: F_n = F_{n-1} + F_{n-2}.
\end{equation}

\section{Fibonacci-type Internal Combinatorics}
\label{sec:fibonacci}

A fundamental feature of block-separated overpartitions is the appearance of the classical Fibonacci numbers in their internal structure. This phenomenon arises
from the fact that, once the distinct part-sizes have been chosen, the
overlining pattern on these blocks is governed by the constraint that no two
consecutive blocks may both be overlined. This turns the decoration process
into a problem of counting binary words with forbidden consecutive \(1\)'s, a
combinatorial problem whose solution is known to be Fibonacci-like.

\paragraph{Binary words avoiding the pattern \(\boldsymbol{11}\).}
Suppose a block-separated overpartition uses exactly \(r\) distinct part-sizes.
Label the corresponding blocks in decreasing order of size as positions
\(\{1,\ldots,r\}\). Each block may be either \emph{plain} or \emph{overlined},
and this choice can be encoded by a binary  word $w = w_1 w_2 \cdots w_r$, with $w_i \in \{0,1\}$, and  where \(w_i=1\) means that the block in position \(i\) is overlined. The
block-separated condition is exactly the constraint
\[
w_i w_{i+1} = 0 \quad\text{for all } 1\le i < r,
\]
i.e.\ the word \(w\) contains no consecutive \(1\)'s.

It is well known (and easy to prove by a standard recurrence argument) that the
number of length-\(r\) binary words with no occurrence of the pattern \(11\) is
the Fibonacci number \(F_{r+2}\). Equivalently, such a word may be viewed as a tiling of a board of length \(r\) by the tiles \(\boxed{0}\) and \(\boxed{10}\), representing plain and overlined blocks. Reading from left to right, every binary word of length \(r\) avoiding the pattern \(11\) admits a unique decomposition into these tiles, and this correspondence immediately implies the Fibonacci recurrence. This
bijection immediately yields the recurrence
\begin{equation}\label{eq:RecurenceFiboShifted}
    a_0 = 1,\ a_1 = 2,\quad \forall r\geq 2:a_r  =  a_{r-1} + a_{r-2},
\end{equation}
which is solved by \(a_r = F_{r+2}\).

\paragraph{Connection with independent sets on a path.}
An equivalent viewpoint is obtained by identifying the block positions
\(\{1,\ldots,r\}\) with the vertices of a path graph \(P_r\). Choosing which
blocks are overlined is equivalent to selecting an \emph{independent set} of
vertices (no two adjacent), because overlined blocks may not appear in
consecutive positions. The number of independent sets of size \(m\) in a path
of length \(r\) is
\[
\binom{r-m+1}{m},
\]
and summing over all \(m\) yields
\begin{equation}
    \sum_{m\ge0} \binom{r-m+1}{m} = F_{r+2}.
\end{equation}
This independent-set interpretation is particularly useful in the symmetric
function expansion given in Theorem \ref{thm:symm-expansion}.

The quantity \(F_{r+2}\) is exactly the number of legal independent subsets of a
path of length \(r\), i.e., the number of admissible decoration patterns for an
\(r\)-block part-size set.

\paragraph{Automaton and Fibonacci polynomials.}
Viewed through the two-state automaton of Figure~\ref{fig:automaton}, the
presence of Fibonacci structure is even more transparent. The automaton has two
states:
\[
0 = \text{``last block plain''},\qquad
1 = \text{``last block overlined''},
\]
and transitions
\[
0 \xrightarrow{\text{plain/absent}} 0,\qquad
0 \xrightarrow{\text{overlined}} 1,\qquad
1 \xrightarrow{\text{plain}} 0,\qquad
1 \xrightarrow{\text{absent}} 1.
\]
If one ignores the \(q\)-weights and only tracks plain/overlined choices
(i.e.\ ignores the ``absent'' transitions), then the transition rules reduce
to exactly the classical automaton that generates binary words avoiding the
pattern \(11\). When \(q\)-weights are reinstated, the automaton becomes a
weight-enriched version of the Fibonacci automaton, and the polynomial
refinement
\begin{equation}
    F_{r+2}(y) = \sum_{m\ge0} \binom{r-m+1}{m} y^m
\end{equation}
appears naturally when we introduce a variable \(y\) marking the number of
overlined blocks.

\paragraph{Consequences for the  generating function.}
The Fibonacci decoration mechanism leads directly to the the following representation.
\begin{theorem}\label{thm:symm-expansion}
We have the symmetric-function expansion
\begin{equation}\label{eq:symm-expansion}
\mathcal{F}(q)=\sum_{r\ge 0} F_{r+2}\, e_r\!\bigl(S_1(q),S_2(q),S_3(q),\dots\bigr),
\end{equation}
where 
$e_r$ is the $r$-th elementary symmetric function:
\[
e_r(x_1,x_2,\dots)=\sum_{1\le i_1<i_2<\cdots<i_r} x_{i_1}x_{i_2}\cdots x_{i_r},
\qquad e_0=1.
\]
\end{theorem}

\begin{proof}
Fix $n\ge 0$. We count block-separated overpartitions of total size $n$ by conditioning on the
number $r$ of distinct part-sizes that appear. (Equivalently, $r$ is the number of
distinct blocks in the block form of the partition.)

Let $\{j_1<j_2<\cdots<j_r\}\subset\mathbb{N}$ be the set of distinct part-sizes used.
For each chosen size $j_t$, the corresponding block consists of $m_t\ge 1$ copies of $j_t$,
contributing weight $q^{m_t j_t}$. Summing over all multiplicities $m_t\ge 1$ gives the factor
\[
S_{j_t}(q)=\sum_{m\ge 1} q^{mj_t}=\frac{q^{j_t}}{1-q^{j_t}}.
\]
Hence, for a fixed set $\{j_1,\dots,j_r\}$, the total contribution of all multiplicity choices is
\[
\prod_{t=1}^r S_{j_t}(q).
\]
Summing over all $r$-subsets of $\mathbb{N}$ yields exactly the elementary symmetric function
\[
\sum_{1\le j_1<\cdots<j_r}\ \prod_{t=1}^r S_{j_t}(q)
= e_r\!\bigl(S_1(q),S_2(q),S_3(q),\dots\bigr).
\]

Once the $r$ distinct part-sizes are fixed, the only remaining choice is which of the $r$ blocks
are overlined. Write the blocks in decreasing order of part-size; this yields an ordered list of
$r$ block positions. The block-separated condition is precisely that no two consecutive block
positions are both overlined. Therefore the admissible overlining patterns are in bijection with
binary words of length $r$ avoiding the factor $11$, where $1$ indicates ``overlined'' and $0$
indicates ``plain''.

It is standard (and follows from the Recurrence \eqref{eq:RecurenceFiboShifted})
that the number of such binary words is $F_{r+2}$. Importantly, this factor depends only on $r$,
not on the chosen part-sizes or multiplicities.

For each fixed $r\ge 0$, the total generating function contribution of block-separated
overpartitions with exactly $r$ distinct part-sizes is thus
\[
F_{r+2}\,e_r\!\bigl(S_1(q),S_2(q),S_3(q),\dots\bigr).
\]
Summing over all $r\ge 0$ gives \eqref{eq:symm-expansion}, completing the proof.
\end{proof}

Thus the global structure of block-separated overpartitions is a convolution of
Euler-type partition structure (via the $S_j(q)$) with Fibonacci-type internal
combinatorics (via the $F_{r+2}$). This hybrid structure explains both the
regularity of the generating function and its strong connections to classical
objects such as independent sets, pattern-avoiding words, and tilings.

\section{Transfer-matrix Formulation with a Two-state Automaton}\label{sec:transfer_matrix}

In this section we give a detailed transfer-matrix description of
block-separated overpartitions, based on the per-size block generating functions
\(S_j(q)\) and a simple two-state automaton that encodes the
constraint ``no two consecutive distinct blocks are both overlined''.

\begin{definition}\label{def:states}
We scan the possible part-sizes in increasing order
\[
j = 1,2,3,\ldots,
\]
and decide, at each \(j\), whether to use the part \(j\) in the partition (and
if so, whether the block of \(j\)'s is overlined or not). We maintain a
\emph{state} in \(\{0,1\}\) defined as follows:
\begin{itemize}
  \item State \(0\): the last \emph{present} block (if any) is \emph{not overlined};
  \item State \(1\): the last present block \emph{is overlined}.
\end{itemize}
Thus state \(1\) represents the situation where the most recent distinct
part-size used so far was overlined, so the next present block is forbidden to be overlined.
\end{definition}

\begin{remark}
If we skip some part-size \(j\) (i.e., no parts of size \(j\) appear), the state
does not change, since no new block is created. Only when we actually use the
part-size \(j\) do we update the state according to whether the block is
overlined or not.
\end{remark}

\begin{proposition}\label{prop:transition}
Fix a part-size \(j\ge1\). At this size, there are three combinatorial choices:
\[
\begin{array}{ll}
\text{\textup{(A)} Block $j$ absent:} &
\text{no parts of size } j \text{ appear},\\[3pt]
\text{\textup{(B)} Block $j$ present, plain:} &
\text{at least one } j \text{ appears, but the first $j$ is not overlined},\\[3pt]
\text{\textup{(C)} Block $j$ present, overlined:} &
\text{at least one } j \text{ appears, and the first $j$ is overlined}.
\end{array}
\]
The contributions of these choices to the \(q\)-weights and the state transitions
are:
\begin{itemize}
  \item[(A)] Block $j$ absent: Weight 1  and  state unchanged.
  \item[(B)] Block $j$ present and plain:  
  The block consists of \(m\ge1\) parts of size \(j\),
  so the total size contribution is \(mj\), and the generating function is $S_j(q)$. 
  
  Since this block is \emph{not} overlined, the new last present block is plain,
  and the next state is \(0\), regardless of the current state.
  \item[(C)] Block $j$ present and  overlined: 
  Again the multiplicity \(m\ge1\) contributes weight \(S_j(q)\).
  However, because the first occurrence is overlined, this choice is allowed
  only if the current state is \(0\), i.e., the previously present block was
  not overlined. The next state is then \(1\).
\end{itemize}
\end{proposition}

\begin{example}\label{ex:local-transitions}
To illustrate, consider the first few part-sizes with truncation in \(q\) for clarity.
\begin{itemize}
  \item At \(j=1\),
  \[
  S_1(q) = q + q^2 + q^3 + \cdots.
  \]
  From state \(0\):
  \begin{itemize}
    \item block \(1\) absent: stay in state \(0\), weight \(1\);
    \item block \(1\) present, plain: go to state \(0\), weight \(S_1(q)\);
    \item block \(1\) present, overlined: go to state \(1\), weight \(S_1(q)\).
  \end{itemize}
  From state \(1\):
  \begin{itemize}
    \item block \(1\) absent: stay in state \(1\), weight \(1\);
    \item block \(1\) present, plain: go to state \(0\), weight \(S_1(q)\);
    \item block \(1\) present, overlined: forbidden (would create two consecutive
    overlined blocks).
  \end{itemize}

  \item At \(j=2\),
  \[
  S_2(q) = q^2 + q^4 + q^6 + \cdots.
  \]
  The same pattern holds, with the understanding that any present block of size
  \(2\) contributes at least \(q^2\) to the size. For example, from state \(1\) we
  may choose:
  \begin{itemize}
    \item block \(2\) absent: stay in state \(1\), weight \(1\);
    \item block \(2\) present, plain: go to state \(0\), weight \(S_2(q)\);
    \item block \(2\) present, overlined: forbidden.
  \end{itemize}
\end{itemize}
In each case, the rule ``no two consecutive blocks overlined'' is enforced by disallowing
the transition from state \(1\) to state \(1\) via an overlined block.

\end{example}

The transitions described in  Proposition~\ref{prop:transition} can be conveniently encoded
in a \(2\times2\) matrix.

\begin{lemma}\label{def:Mj}
Let \(M_j(q)\) be the matrix whose entry \(M_j(a,b)\) is the total generating
function weight of all ways to move from state \(a\) at size \(j-1\) to state \(b\)
after processing choices at size \(j\). Then:
\begin{equation}
    M_j(q) =
\begin{pmatrix}
M_j(0,0) & M_j(0,1) \\[3pt]
M_j(1,0) & M_j(1,1)
\end{pmatrix}=
\begin{pmatrix}
\dfrac{1}{1-q^j} & \dfrac{q^j}{1-q^j} \\[5pt]
\dfrac{q^j}{1-q^j} & 1
\end{pmatrix}.
\end{equation}
\end{lemma}
\begin{proof}
    By Proposition~\ref{prop:transition} we obtain:
\begin{itemize}
  \item From state \(0\) to state \(0\): either block \(j\) is absent
  (weight \(1\)) or present but plain (weight \(S_j(q)\)):
  \[
  M_j(0,0) = 1+S_j(q) = \frac{1}{1-q^j}.
  \]
  \item From state \(0\) to state \(1\): block \(j\) must be present and overlined
  (weight \(S_j(q)\)):
  \[
  M_j(0,1) = S_j(q) = \frac{q^j}{1-q^j}.
  \]
  \item From state \(1\) to state \(0\): the only possibility is a present, plain block
  (weight \(S_j(q)\)):
  \[
  M_j(1,0) = S_j(q) = \frac{q^j}{1-q^j}.
  \]
  \item From state \(1\) to state \(1\): the only possibility is absence of block \(j\)
  (weight \(1\)):
  \[
  M_j(1,1) = 1.
  \]
\end{itemize}
Therefore
\[
M_j(q) =
\begin{pmatrix}
1 + S_j(q) & S_j(q) \\[3pt]
S_j(q) & 1
\end{pmatrix}
=
\begin{pmatrix}
\dfrac{1}{1-q^j} & \dfrac{q^j}{1-q^j} \\[5pt]
\dfrac{q^j}{1-q^j} & 1
\end{pmatrix}.\qedhere
\]
\end{proof}

Figure~\ref{fig:automaton} provides a  visual encoding of all permitted
local transitions at a given part-size \(j\). The two nodes represent the
machine states defined in Definition~\ref{def:states}:
\begin{itemize}
  \item The node labelled \(0\) represents the situation in which the last
  present block (if any) is plain (not overlined). This is the ``safe''
  state from which both plain and overlined choices are allowed at size~\(j\).
  \item The node labelled \(1\) represents the situation in which the last
  present block is overlined. From this state, placing another overlined
  block at size \(j\) is forbidden, because it would violate the defining
  condition of block-separated overpartitions.
\end{itemize}

Each directed arrow has a label indicating the combinatorial choice at size
\(j\) together with its generating function weight.

\textbf{Loops}:  
\begin{itemize}
  \item The loop at state \(0\) has two contributions:
  \begin{itemize}
    \item \emph{absent block}: no parts of size \(j\) appear, weight \(1\);
    \item \emph{plain block}: a block of size \(j\) appears with multiplicity
    \(\ge1\), weight \(S_j(q)\). Because this block is plain, the new state
    remains \(0\).
  \end{itemize}
  \item The loop at state \(1\) has only one contribution:
  \[
  \text{absent block: weight } 1.
  \]
  Indeed, if the previous block was overlined, then placing a plain block at
  size \(j\) would move to state \(0\), and placing an overlined block would be
  forbidden. Therefore, the only way to remain in state \(1\) is to skip size
  \(j\).
\end{itemize}

\textbf{Arcs}:  
\begin{itemize}
  \item The arc \(0 \to 1\) corresponds to choosing an overlined block
  at size \(j\). This is allowed only from state \(0\), because the last present
  block must not be overlined. The weight is \(S_j(q)\), the generating function
  for a block of size \(j\) with multiplicity \(\ge1\).
  \item The arc \(1 \to 0\) corresponds to choosing a plain block at size
  \(j\). This is the only way to ``reset'' from an overlined-block state back to
  a plain-block state. The weight is again \(S_j(q)\).
\end{itemize}

Altogether, the automaton encodes precisely the local combinatorics required at
each part-size \(j\): the choice of absent/plain/overlined block, the correct
\(q\)-weights through the factors \(S_j(q)\), and the forbidden transition that
ensures that no two consecutive blocks are overlined. By scanning sizes
\(j=1,2,3,\ldots\), the global structure of block-separated overpartitions is
recovered via the transfer-matrix product of Theorem~\ref{thm:matrix-detailed}.

\begin{figure}[H]
    \centering

\tikzset{every picture/.style={line width=0.75pt}} %set default line width to 0.75pt        

\begin{tikzpicture}[x=0.75pt,y=0.75pt,yscale=-1,xscale=1]

%Shape: Circle [id:dp5285107686606212] 
\draw   (217,132.5) .. controls (217,122.84) and (224.84,115) .. (234.5,115) .. controls (244.16,115) and (252,122.84) .. (252,132.5) .. controls (252,142.16) and (244.16,150) .. (234.5,150) .. controls (224.84,150) and (217,142.16) .. (217,132.5) -- cycle ;

%Shape: Circle [id:dp8320390387287899] 
\draw   (417,132.5) .. controls (417,122.84) and (424.84,115) .. (434.5,115) .. controls (444.16,115) and (452,122.84) .. (452,132.5) .. controls (452,142.16) and (444.16,150) .. (434.5,150) .. controls (424.84,150) and (417,142.16) .. (417,132.5) -- cycle ;

%Curve Lines [id:da4426810034187221] 
\draw    (234.5,115) .. controls (273.7,85.6) and (394.84,87.59) .. (432.31,113.39) ;
\draw [shift={(434.5,115)}, rotate = 218.23] [fill={rgb, 255:red, 0; green, 0; blue, 0 }  ][line width=0.08]  [draw opacity=0] (7.14,-3.43) -- (0,0) -- (7.14,3.43) -- (4.74,0) -- cycle    ;
%Curve Lines [id:da992900271907903] 
\draw    (237,151.76) .. controls (279.48,179.88) and (400.51,176.77) .. (434.5,150) ;
\draw [shift={(234.5,150)}, rotate = 36.87] [fill={rgb, 255:red, 0; green, 0; blue, 0 }  ][line width=0.08]  [draw opacity=0] (7.14,-3.43) -- (0,0) -- (7.14,3.43) -- (4.74,0) -- cycle    ;
%Curve Lines [id:da9905165275218546] 
\draw    (217,129.5) .. controls (162.65,134.47) and (192.3,82.22) .. (224.1,115.43) ;
\draw [shift={(226.05,117.57)}, rotate = 229.09] [fill={rgb, 255:red, 0; green, 0; blue, 0 }  ][line width=0.08]  [draw opacity=0] (7.14,-3.43) -- (0,0) -- (7.14,3.43) -- (4.74,0) -- cycle    ;
%Curve Lines [id:da5808596398008536] 
\draw    (452.55,135.34) .. controls (506.89,140.31) and (479.99,84.86) .. (448.31,117.93) ;
\draw [shift={(446.36,120.07)}, rotate = 310.91] [fill={rgb, 255:red, 0; green, 0; blue, 0 }  ][line width=0.08]  [draw opacity=0] (7.14,-3.43) -- (0,0) -- (7.14,3.43) -- (4.74,0) -- cycle    ;

% Text Node
\draw (228.5,124.9) node [anchor=north west][inner sep=0.75pt]  []  {$0$};
% Text Node
\draw (428.5,124.9) node [anchor=north west][inner sep=0.75pt]  []  {$1$};
% Text Node
\draw (105,88) node [anchor=north west][inner sep=0.75pt]  [font=\footnotesize,] [align=left] {absent or plain :\\ $\displaystyle 1$or $\displaystyle S_{j}( q)$};
% Text Node
\draw (295,74) node [anchor=north west][inner sep=0.75pt]  [font=\footnotesize,] [align=left] {overlined: $\displaystyle S_{j}( q)$};
% Text Node
\draw (308,182) node [anchor=north west][inner sep=0.75pt]  [font=\footnotesize,] [align=left] {plain: $\displaystyle S_{j}( q)$};
% Text Node
\draw (466,90) node [anchor=north west][inner sep=0.75pt]  [font=\footnotesize,] [align=left] {absent: $\displaystyle 1$};

\end{tikzpicture}    
\caption{Two-state automaton governing local transitions at part-size \(j\).}
\label{fig:automaton}
\end{figure}
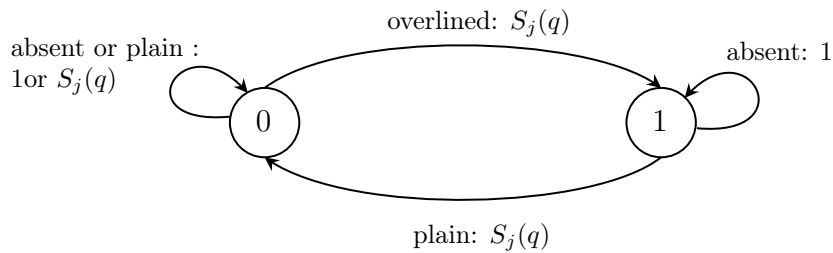

\section{Generating Function}\label{sec:generating_function}
We now assemble the local combinatorial transitions into a global generating function. 
The transfer-matrix formulation allows us to encode all admissible configurations across part-sizes $j=1,2,3,\dots$ via an infinite matrix product. 
This leads to a compact expression for the ordinary generating function $F(q)$.

The global generating function is obtained by composing all local transitions
for \(j=1,2,3,\ldots\). We start with no blocks and hence in state \(0\), encoded
by the row vector \((1,0)\). After processing all sizes, we sum over both
possible final states, encoded by the column vector \( \begin{pmatrix}1\\1
\end{pmatrix}\).

\begin{theorem}\label{thm:matrix-detailed}
The ordinary generating function of block-separated overpartitions is
\begin{equation}
    \mathcal{F}(q) =
\begin{pmatrix} 1 & 0 \end{pmatrix}
\left(\prod_{j\ge1} M_j(q)\right)
\begin{pmatrix} 1 \\ 1 \end{pmatrix},
\end{equation}
where \(M_j(q)\) is the transfer matrix of Lemma~\ref{def:Mj}.
\end{theorem}

\begin{proof}
By construction, for each fixed size \(j\), the matrix \(M_j(q)\) encodes all
legal transitions between states \(0\) and \(1\) with the correct \(q\)-weights:
each allowed choice (absent, present plain, present overlined) contributes a
monomial \(q^{\text{contribution}}\) multiplied by the generating function
\(S_j(q)\) when the block is present, and the block-separated condition is
enforced by disallowing the transition \(1\to1\) via an overlined block.

The product \(\prod_{j\ge1} M_j(q)\) is taken in increasing order of \(j\), and
matrix multiplication corresponds to concatenating choices at each size. The
row vector \((1,0)\) represents the unique starting configuration before any
part-size is considered: no blocks have been chosen and the last present block
is (vacuously) plain. Multiplying from the left by \((1,0)\) propagates all
legal sequences of choices across all sizes. Finally, multiplying on the right
by \(\begin{pmatrix}1\\1\end{pmatrix}\) sums the contributions from
both possible final states, since the final state (whether the last block is
overlined or not) does not matter for the total count. Thus the coefficient
of \(q^n\) in
\[
\begin{pmatrix} 1 & 0 \end{pmatrix}
\left(\prod_{j\ge1} M_j(q)\right)
\begin{pmatrix} 1 \\ 1 \end{pmatrix}
\]
is precisely the number of block-separated overpartitions of \(n\).
\end{proof}

\section{Euler Factorization}\label{sec:Euler}
A central structural feature of the transfer matrix is the presence of the classical Euler factor $(1-q^j)^{-1}$ in its upper-left entry. 
Extracting these factors leads to a normalized formulation and reveals a clean Euler-type factorization of the generating function.

\begin{theorem}\label{thm:euler-factorization} We have 
\begin{equation}
    \mathcal{F}(q)
= \frac{1}{(q)_\infty}
\left(
\widehat{F}_0^{(\infty)}(q) + \widehat{F}_1^{(\infty)}(q)
\right), 
\end{equation}
where the modified sequences $\widehat{F}_0^{(n)}$ and $\widehat{F}_1^{(n)}$ satisfy the following coupled recurrences  
\begin{align}
\widehat{F}_0^{(n)} &= \widehat{F}_0^{(n-1)} + q^n\widehat{F}_1^{(n-1)},\label{eq:FirstRec}\\
\widehat{F}_1^{(n)} &= q^n \widehat{F}_0^{(n-1)} + (1-q^n)\widehat{F}_1^{(n-1)},\label{eq:SecondRec}
\end{align}
with initial conditions \(\widehat{F}_0^{(0)}=1,\ \widehat{F}_1^{(0)}=0.\)
\end{theorem}
\begin{proof}

% {\color{red}
% \begin{align*}
% \widehat{F}_0^{(n)}-\widehat{F}_1^{(n-1)} &= \widehat{F}_0^{(n-1)} ,\\
% \widehat{F}_1^{(n)} &= q^n \widehat{F}_0^{(n)} + (1-2q^n)\widehat{F}_1^{(n-1)},
% \end{align*}
% }

The upper-left entry of the transfer matrix $M_j(q)$ is simply a   classical Euler factor $\dfrac{1}{1-q^j}$, which is precisely the \(j\)-th factor in the Euler product (or the infinite $q$-Pochhammer symbol) \eqref{eq:qpoch}.

Extracting these factors from the infinite product
\(\prod_{j\ge1} M_j(q)\) leads to a clean normalization of the transfer-matrix
recurrence.

To make this precise, define the \emph{normalized} matrices
\begin{equation}\label{eq:NormlizedMatrix}
    \widehat{M}_j(q)
 = 
(1-q^j)\, M_j(q)
=
\begin{pmatrix}
1 & q^j \\[4pt]
q^j & 1-q^j
\end{pmatrix}.
\end{equation}
Then for any row vector \(V=(V_0,V_1)\) we have
\[
V M_j(q)
=
\frac{1}{1-q^j}\, V \widehat{M}_j(q).
\]
It follows that
\[
(1,0)\!\left(\prod_{j\ge1} M_j(q)\right)
=
\left(\prod_{j\ge1} \frac{1}{1-q^j}\right)
(1,0)\!\left(\prod_{j\ge1} \widehat{M}_j(q)\right)
=
\frac{1}{(q)_\infty}\,
(1,0)\!\left(\prod_{j\ge1} \widehat{M}_j(q)\right).
\]

Let \(\bigl(\widehat{F}_0^{(n)}(q),\widehat{F}_1^{(n)}(q)\bigr)\) denote the row vector
obtained after multiplying the normalized matrices \(\widehat{M}_1(q),\ldots,\widehat{M}_n(q)\), that is, the \emph{normalized row-vectors} 
\begin{equation}\label{eq:NormlizedVector}
    \bigl(\widehat{F}_0^{(n)},\widehat{F}_1^{(n)}\bigr)
=
(1,0) \,\widehat{M}_1(q)\,\widehat{M}_2(q)\cdots\widehat{M}_n(q).
\end{equation}
Therefore, we have the recurrence
\begin{equation}
    \bigl(\widehat{F}_0^{(n)},\widehat{F}_1^{(n)}\bigr)
=
\bigl(\widehat{F}_0^{(n-1)},\widehat{F}_1^{(n-1)}\bigr)
\begin{pmatrix}
1 & q^n \\[4pt]
q^n & 1-q^n
\end{pmatrix},
\end{equation}
which yields the component-wise system
\begin{align*}
\widehat{F}_0^{(n)}
&=
\widehat{F}_0^{(n-1)} + q^n\,\widehat{F}_1^{(n-1)},\\[4pt]
\widehat{F}_1^{(n)}
&=
q^n\,\widehat{F}_0^{(n-1)} + (1-q^n)\,\widehat{F}_1^{(n-1)}.
\end{align*}
The initial conditions come from \((1,0)\) before any matrices are applied:
\[
\widehat{F}_0^{(0)} = 1,\qquad
\widehat{F}_1^{(0)} = 0.
\]

Finally, summing over both terminal states gives
\[
\mathcal{F}(q)
=
\begin{pmatrix}1&0\end{pmatrix}
\!\left(\prod_{j\ge1} M_j(q)\right)
\begin{pmatrix}1\\1\end{pmatrix}
=
\frac{1}{(q)_\infty}\,
\Bigl(\widehat{F}_0^{(\infty)}(q)+\widehat{F}_1^{(\infty)}(q)\Bigr),
\]
proving the normalized Euler-type factorization.
\end{proof}

\begin{example}[Step-by-step multiplication for \(j=1,2,3\)]\label{ex:M123}
To see Theorem~\ref{thm:matrix-detailed} in action, we truncate the series to
order \(q^3\) and explicitly compute the contributions of part-sizes
\(j=1,2,3\).

First, we truncate
\[
S_1(q) = q + q^2 + q^3 + O(q^4),\qquad
S_2(q) = q^2 + O(q^4),\qquad
S_3(q) = q^3 + O(q^4),
\]
and form
\[
M_1(q) =
\begin{pmatrix}
1+S_1(q) & S_1(q)\\
S_1(q) & 1
\end{pmatrix}
=
\begin{pmatrix}
1+q+q^2+q^3 & q+q^2+q^3\\
q+q^2+q^3 & 1
\end{pmatrix}
+ O(q^4),
\]
\[
M_2(q) =
\begin{pmatrix}
1+S_2(q) & S_2(q)\\
S_2(q) & 1
\end{pmatrix}
=
\begin{pmatrix}
1+q^2 & q^2\\
q^2 & 1
\end{pmatrix}
+ O(q^4),
\qquad
M_3(q) =
\begin{pmatrix}
1+q^3 & q^3\\
q^3 & 1
\end{pmatrix}
+ O(q^4).
\]

\begin{itemize}
    \item Step 1: process size \(j=1\). 
    Starting from \((1,0)\), we obtain
\[
(1,0)\,M_1(q)
=
\bigl(1+S_1(q),\,S_1(q)\bigr)
=
\bigl(1+q+q^2+q^3, q+q^2+q^3\bigr) + O(q^4).
\]
The first component corresponds to all configurations after considering size \(1\)
and ending in state \(0\); the second component corresponds to those ending in
state \(1\).
\item Step 2: process size \(j=2\). Now multiply by \(M_2(q)\):
\[
(1,0)\,M_1(q)\,M_2(q)
=
\bigl(F_0^{(2)}(q),\,F_1^{(2)}(q)\bigr),
\]
where, after truncation at \(q^3\),
\[
F_0^{(2)}(q) = 1 + q + 2q^2 + 3q^3 + O(q^4),\qquad
F_1^{(2)}(q) = q + 2q^2 + 2q^3 + O(q^4).
\]
\item Step 3: process size \(j=3\). 
Multiplying once more by \(M_3(q)\) gives
\[
(1,0)\,M_1(q)\,M_2(q)\,M_3(q)
=
\bigl(F_0^{(3)}(q),\,F_1^{(3)}(q)\bigr),
\]
with
\[
F_0^{(3)}(q) = 1 + q + 2q^2 + 4q^3 + O(q^4),\qquad
F_1^{(3)}(q) = q + 2q^2 + 3q^3 + O(q^4).
\]
Finally, summing over the two terminal states yields the truncated global
generating function:
\[
\bigl(F_0^{(3)}(q) + F_1^{(3)}(q)\bigr)
=
1 + 2q + 4q^2 + 7q^3 + O(q^4).
\]
Thus we recover the initial coefficients
\[
b(0)=1,\quad b(1)=2,\quad b(2)=4,\quad b(3)=7.
\]

\end{itemize}

\end{example}

By direct computation from the recurrence we obtain
\[
\mathcal{F}(q) =
1 + 2q + 4q^2 + 7q^3 + 12q^4 + 19q^5 + 31q^6 + 47q^7 + 72q^8 + 107q^9 + 157q^{10} + \cdots
\]

\section{Recurrence Structure}\label{sec:Recurrence}
The normalized formulation yields coupled first-order recurrences for the state-polynomials \eqref{eq:FirstRec} and \eqref{eq:SecondRec}. 
In this section we uncouple these relations and derive corresponding second-order scalar recurrences.

Let $\widehat F^{(n)}_0(q)$ and $\widehat F^{(n)}_1(q)$ be defined respectively by Relations
\eqref{eq:FirstRec} and \eqref{eq:SecondRec}, with $\widehat F^{(0)}_0=1$ and $\widehat F^{(0)}_1=0$. For $n \ge 0$, set
\[
\alpha_n(q) := \widehat F^{(n)}_0(q),
\qquad
\beta_n(q) := \widehat F^{(n)}_1(q).
\]
For $n \ge 1$, define
\[
d_n(q) := 1 + q - q^n,
\qquad
c_n(q) := q^{2n-1} + q^n - q.
\]

\begin{proposition}\label{prop:uncouple}
We have $\{\alpha_n(q)\}_{n\ge0}$ satisfies the homogeneous second–order recurrence
\begin{equation}\label{eq:a-second}
\alpha_0(q)=1,  \alpha_1(q)=1,\quad \forall n\geq 2:  \alpha_n(q)=d_n(q)\alpha_{n-1}(q)+c_n(q)\alpha_{n-2}(q).
\end{equation}
Moreover, for all $n\geq 0$, we have
\begin{equation}\label{eq:F1-from-a}
\beta_n(q)=\frac{\alpha_{n+1}(q)-\alpha_n(q)}{q^{n+1}}.
\end{equation}
\end{proposition}

\begin{proof}
From Relation \eqref{eq:FirstRec}   we have, for $n\ge 1$,
\[
\beta_{n-1}=\frac{\alpha_n -\alpha_{n-1} }{q^n}.
\]
Apply Relation \eqref{eq:SecondRec} with index $n-1$ and substitute
$ \beta_{n-2}=(\alpha_{n-1}-\alpha_{n-2})/q^{n-1}$ to obtain
\[
\beta_{n-1}
=q^{n-1}\alpha_{n-2}+\frac{1-q^{n-1}}{q^{n-1}}\,(\alpha_{n-1}-\alpha_{n-2}).
\]
Finally, plug this into $\alpha_n=\alpha_{n-1}+q^n\beta_{n-1}$ (Relation \eqref{eq:FirstRec}) and simplify; this gives \eqref{eq:a-second}.
Formula \eqref{eq:F1-from-a} is just the rearrangement of
$\alpha_{n+1}=\alpha_n+q^{n+1}  \beta_{n}$.
\end{proof}

\begin{proposition}\label{prop:beta-homogeneous-shifted}
 
We have $\{\beta_n(q)\}_{n\ge0}$ satisfies the homogeneous second–order recurrence
\begin{equation}\label{eq:beta-rec}
\beta_0(q)=0,\
\beta_1(q)=q,\quad\forall n\geq 2: \beta_{n}(q)
=
d_n(q)\beta_{n-1}(q)
+c_n(q)\beta_{n-2}(q).
\end{equation}
\end{proposition}

\begin{proof}
 
Using the recurrence for $\alpha_{n+2}$, given by \eqref{eq:a-second},  we compute
\begin{align*}
\beta_{n+1}
&=\frac{\alpha_{n+2}-\alpha_{n+1}}{q^{n+2}}
\\&=\frac{(d_{n+2}-1)\alpha_{n+1}+c_{n+2}\alpha_n}{q^{n+2}}\\
&=\Bigl(q^{-(n+1)}-1\Bigr)\alpha_{n+1}+q^{-(n+2)}c_{n+2}\alpha_n,
\end{align*}
since $d_{n+2}-1=q-q^{n+2}$.

From \eqref{eq:F1-from-a} we have $\alpha_n=\alpha_{n+1}-q^{n+1}\beta_n$, hence
\begin{align}
\beta_{n+1}
&=\Bigl(q^{-(n+1)}-1\Bigr)\alpha_{n+1}
+q^{-(n+2)}c_{n+2}\bigl(\alpha_{n+1}-q^{n+1}\beta_n\bigr)\notag\\
&=\Bigl((q^{-(n+1)}-1)+q^{-(n+2)}c_{n+2}\Bigr)\alpha_{n+1}-q^{-1}c_{n+2}\beta_n.
\label{eq:beta-mid}
\end{align}
Now use the identity
\[
(q-q^{n+2})+c_{n+2}=q^{2n+3}
\quad\Longrightarrow\quad
(q^{-(n+1)}-1)+q^{-(n+2)}c_{n+2}=q^{n+1},
\]
to simplify \eqref{eq:beta-mid} into
\begin{equation}\label{eq:beta-mid2}
\beta_{n+1}=q^{n+1}\alpha_{n+1}-q^{-1}c_{n+2}\beta_n.
\end{equation}
Apply the recurrence for $\alpha_{n+1}$, given by \eqref{eq:a-second}, and  substitute $\alpha_n=\alpha_{n+1}-q^{n+1}\beta_n$, and  $\alpha_{n-1}=\alpha_n-q^n\beta_{n-1}=\alpha_{n+1}-q^{n+1}\beta_n-q^n\beta_{n-1}$ (which follow from \eqref{eq:F1-from-a} for $n$ and $n-1$). This yields
\[
\alpha_{n+1}=(d_{n+1}+c_{n+1})\alpha_{n+1}-q^{n+1}(d_{n+1}+c_{n+1})\beta_n-c_{n+1}q^n\beta_{n-1}.
\]
Rearranging and using $d_{n+1}+c_{n+1}=1+q^{2n+1}$, and $1-d_{n+1}-c_{n+1}=-q^{2n+1}$, we obtain
\begin{equation}\label{eq:an1-in-beta}
\alpha_{n+1}=q^{-n}(1+q^{2n+1})\beta_n+c_{n+1}q^{-(n+1)}\beta_{n-1}.
\end{equation}
From \eqref{eq:beta-mid2} and \eqref{eq:an1-in-beta},
\begin{align*}
\beta_{n+1}
&=q^{n+1}\Bigl(q^{-n}(1+q^{2n+1})\beta_n+c_{n+1}q^{-(n+1)}\beta_{n-1}\Bigr)
-q^{-1}c_{n+2}\beta_n\\
&=\Bigl(q(1+q^{2n+1})-q^{-1}c_{n+2}\Bigr)\beta_n+c_{n+1}\beta_{n-1}.
\end{align*}
Finally, since $q^{-1}c_{n+2}=q^{2n+3}+q^{n+2}-1$, we have
\[
q(1+q^{2n+1})-q^{-1}c_{n+2}=1+q-q^{n+1},
\]
and therefore
\[
\beta_{n+1}=(1+q-q^{n+1})\beta_n+c_{n+1}\beta_{n-1}.
\]
Replacing $c_{n+1}$ by its expression $q^{2n+1}+q^{n+1}-q$ gives \eqref{eq:beta-rec}.

Regarding the initial values, we have  $\beta_0=\dfrac{\alpha_1-\alpha_0}{q}=0$, and 
\[
\alpha_2=(1+q-q^2)\alpha_1+(q^3+q^2-q)\alpha_0=1+q^3
\quad\Longrightarrow\quad
\beta_1=\frac{\alpha_2-\alpha_1}{q^2}=q.\qedhere
\]
\end{proof}

 Tables~\ref{tab:F0hat-coeff} and~\ref{tab:F1hat-coeff} 
list the coefficients of the polynomials 
$\alpha_n(q)=\widehat F^{(n)}_0(q)$ and $\beta_n(q)=\widehat F^{(n)}_1(q)$ 
for $0\le n\le 10$, truncated at $q^{12}$. As we can see, the degrees grow roughly quadratically in $n$, 
reflecting the accumulation of contributions from the factors $q^j$ 
in the transfer-matrix product.
% \begin{table}[H]
% \centering
% \caption{First values of $\widehat F^{(n)}_0(q)$ for $0\le n\le 10$.}
% \begin{tabular}{c|l}
% $n$ & $\widehat F^{(n)}_0(q)$ \\ \hline
% 0 & $1$ \\[3pt]
% 1 & $1$ \\[3pt]
% 2 & $1 + q^2$ \\[3pt]
% 3 & $1 + q^2 + q^3$ \\[3pt]
% 4 & $1 + q^2 + q^3 + q^4 + q^6$ \\[3pt]
% 5 & $1 + q^2 + q^3 + q^4 + q^5 + q^6 + q^7 + q^9$ \\[3pt]
% 6 & $1 + q^2 + q^3 + q^4 + q^5 + 2q^6 + q^7 + q^8 + q^9 + q^{10} + q^{12}$ \\[3pt]
% 7 & $1 + q^2 + q^3 + q^4 + q^5 + 2q^6 + 2q^7 + q^8 + 2q^9 + q^{10} + q^{11} + q^{12} + q^{13} + q^{15}$ \\[3pt]
% 8 & $1 + q^2 + q^3 + q^4 + q^5 + 2q^6 + 2q^7 + 2q^8 + 2q^9 + 2q^{10} + q^{11} + 2q^{12} + q^{13} + q^{14} + q^{15} + q^{16} + q^{18}$ \\[3pt]
% 9 & (polynomial of degree $20$; omitted for brevity) \\[3pt]
% 10 & (polynomial of degree $22$; omitted for brevity)
% \end{tabular}
% \end{table}

% \begin{table}[H]
% \centering
% \caption{First values of $\widehat F^{(n)}_1(q)$ for $0\le n\le 10$.}
% \begin{tabular}{c|l}
% $n$ & $\widehat F^{(n)}_1(q)$ \\ \hline
% 0 & $0$ \\[3pt]
% 1 & $q$ \\[3pt]
% 2 & $q + q^2 - q^3$ \\[3pt]
% 3 & $q + q^2 + q^3 - q^4 - q^5$ \\[3pt]
% 4 & $q + q^2 + q^3 + q^4 + q^5 - q^6 - q^7$ \\[3pt]
% 5 & $q + q^2 + q^3 + q^4 + q^5 + q^6 + q^7 - q^8 - q^9$ \\[3pt]
% 6 & $q + q^2 + q^3 + q^4 + q^5 + 2q^6 + q^7 + q^8 - q^9 - q^{10}$ \\[3pt]
% 7 & (polynomial of degree $15$; omitted for brevity) \\[3pt]
% 8 & (polynomial of degree $18$; omitted for brevity) \\[3pt]
% 9 & \dots \\[3pt]
% 10 & \dots
% \end{tabular}
% \end{table}

\begin{table}[H]

\begin{tabular*}{\linewidth}{@{\extracolsep{\fill}}l|ccccccccccccc@{}}
\toprule
$n$ & $q^0$ & $q^1$ & $q^2$ & $q^3$ & $q^4$ & $q^5$ & $q^6$ & $q^7$ & $q^8$ & $q^9$ & $q^{10}$ & $q^{11}$ & $q^{12}$ \\
\midrule
0  & 1 & 0 & 0 & 0 & 0 & 0 & 0 & 0 & 0 & 0 & 0 & 0 & 0 \\
1  & 1 & 0 & 0 & 0 & 0 & 0 & 0 & 0 & 0 & 0 & 0 & 0 & 0 \\
2  & 1 & 0 & 0 & 1 & 0 & 0 & 0 & 0 & 0 & 0 & 0 & 0 & 0 \\
3  & 1 & 0 & 0 & 1 & 1 & 1 & -1 & 0 & 0 & 0 & 0 & 0 & 0 \\
4  & 1 & 0 & 0 & 1 & 1 & 2 & 0 & 0 & -1 & -1 & 2 & 0 & 0 \\
5  & 1 & 0 & 0 & 1 & 1 & 2 & 1 & 1 & -1 & -1 & 0 & 1 & 1 \\
6  & 1 & 0 & 0 & 1 & 1 & 2 & 1 & 2 & 0 & -1 & 0 & 0 & 1 \\
7  & 1 & 0 & 0 & 1 & 1 & 2 & 1 & 2 & 1 & 0 & 0 & 0 & 0 \\
8  & 1 & 0 & 0 & 1 & 1 & 2 & 1 & 2 & 1 & 1 & 1 & 0 & 0 \\
9  & 1 & 0 & 0 & 1 & 1 & 2 & 1 & 2 & 1 & 1 & 2 & 1 & 0 \\
10 & 1 & 0 & 0 & 1 & 1 & 2 & 1 & 2 & 1 & 1 & 2 & 2 & 1 \\
\bottomrule
\end{tabular*}
\caption{Coefficients $u_{n,k}$ in $\widehat F^{(n)}_0(q)=\sum_{k\ge0} u_{n,k}q^k$ (truncated at $q^{12}$).}\label{tab:F0hat-coeff}
\end{table}

\begin{table}[H]

\begin{tabular*}{\linewidth}{@{\extracolsep{\fill}}l|ccccccccccccc}
\toprule
$n$ & $q^0$ & $q^1$ & $q^2$ & $q^3$ & $q^4$ & $q^5$ & $q^6$ & $q^7$ & $q^8$ & $q^9$ & $q^{10}$ & $q^{11}$ & $q^{12}$ \\
\midrule
0  & 0 & 0 & 0 & 0 & 0 & 0 & 0 & 0 & 0 & 0 & 0 & 0 & 0 \\
1  & 0 & 1 & 0 & 0 & 0 & 0 & 0 & 0 & 0 & 0 & 0 & 0 & 0 \\
2  & 0 & 1 & 1 & -1 & 0 & 0 & 0 & 0 & 0 & 0 & 0 & 0 & 0 \\
3  & 0 & 1 & 1 & 0 & -1 & -1 & 2 & 0 & 0 & 0 & 0 & 0 & 0 \\
4  & 0 & 1 & 1 & 0 & 0 & -2 & 1 & 1 & 2 & 2 & -3 & 0 & 0 \\
5  & 0 & 1 & 1 & 0 & 0 & -1 & 0 & 0 & 3 & 3 & 1 & -1 & -1 \\
6  & 0 & 1 & 1 & 0 & 0 & -1 & 1 & -1 & 2 & 4 & 2 & 2 & 0 \\
7  & 0 & 1 & 1 & 0 & 0 & -1 & 1 & 0 & 1 & 3 & 3 & 3 & 3 \\
8  & 0 & 1 & 1 & 0 & 0 & -1 & 1 & 0 & 2 & 2 & 2 & 4 & 4 \\
9  & 0 & 1 & 1 & 0 & 0 & -1 & 1 & 0 & 2 & 3 & 1 & 3 & 5 \\
10 & 0 & 1 & 1 & 0 & 0 & -1 & 1 & 0 & 2 & 3 & 2 & 2 & 4 \\
\bottomrule
\end{tabular*}
\caption{Coefficients $v_{n,k}$ in $\widehat F^{(n)}_1(q)=\sum_{k\ge0} v_{n,k}q^k$ (truncated at $q^{12}$).}\label{tab:F1hat-coeff}
\end{table}

Note that the normalized polynomials  $\alpha_n(q)$ and  $\beta_n(q)$ may have negative coefficients (as shown in Tables \ref{tab:F0hat-coeff} and \ref{tab:F1hat-coeff}); this is a consequence of extracting the Euler product.

\section{Determinantal  Representation}\label{sec:Determinantel}
Second-order recurrences with variable coefficients naturally admit determinantal representations in terms of continuants. 
We now show that both $\alpha_n(q)$ and $\beta_n(q)$ admit explicit tridiagonal determinant formulas.

\begin{proposition}\label{prop:an-determinant}

For all $m\geq 0$, define the $m\times m$ tridiagonal matrix
\begin{equation}\label{eq:MatriX}
    C_m(q):=
\begin{pmatrix}
d_1(q) & 1      &        &        & 0\\
-c_2(q)& d_2(q) & 1      &        &  \\
       & -c_3(q)& d_3(q) & \ddots &  \\
       &        & \ddots & \ddots & 1\\
0      &        &        & -c_m(q)& d_m(q)
\end{pmatrix}.
\end{equation}
Then, for all $n\geq0$,  the  sequence $\{\alpha_n(q)\}_{n\ge 0}$ 
is given by the closed explicit formula
\begin{equation}
    \alpha_n(q)=\det C_n(q),
\end{equation}
where by convention $\det C_0(q)=1$.
\end{proposition}

\begin{proof}
Set $K_n(q):=\det C_n(q)$ with $K_0(q):=1$. Expanding $\det C_n(q)$ along the last row
(or using the standard continuant recurrence for tridiagonal determinants) yields, for $n\ge 2$,
\[
K_n(q)=d_n(q)\,K_{n-1}(q)-\bigl(-c_n(q)\bigr)\,K_{n-2}(q)
      =d_n(q)\,K_{n-1}(q)+c_n(q)\,K_{n-2}(q).
\]
Moreover,
\[
K_1(q)=\det[d_1(q)]=d_1(q)=1+q-q=1.
\]
Thus $\{K_n(q)\}$ satisfies the same recurrence and initial conditions as $\{\alpha_n(q)\}$, hence
$K_n(q)=\alpha_n(q)$ for all $n\ge 0$.
\end{proof}

\begin{remark}
Equivalently, in first-order form one may write
\[
\binom{\alpha_n(q)}{\alpha_{n-1}(q)}
=
\begin{pmatrix}
d_n(q) & c_n(q)\\
1      & 0
\end{pmatrix}
\binom{\alpha_{n-1}(q)}{\alpha_{n-2}(q)}
=
\left(\prod_{k=2}^n
\begin{pmatrix}
1+q-q^k & q^{2k-1}+q^k-q\\
1 & 0
\end{pmatrix}\right)\binom{1}{1},
\]
so $\alpha_n(q)$ is also the (unique) continuant associated with the coefficients
$\{d_k(q)\},\{c_k(q)\}$.
As a check, at $q=1$ the recurrence becomes $\alpha_n=\alpha_{n-1}+\alpha_{n-2}$, hence
$\alpha_n(1)=F_{n+1}$.
\end{remark}

\begin{proposition}\label{prop:beta-determinant}
For all $m\geq 1$, define the $m\times m$ tridiagonal matrix
\begin{equation}
    D_m(q):=
\begin{pmatrix}
d_2(q) & 1      &        &        & 0\\
-c_3(q)& d_3(q) & 1      &        &  \\
       & -c_4(q)& d_4(q) & \ddots &  \\
       &        & \ddots & \ddots & 1\\
0      &        &        & -c_{m+1}(q)& d_{m+1}(q)
\end{pmatrix}.
\end{equation}
with the convention $\det D_0(q)=1$.

Then, for all $n\geq1$,  the  sequence $\{\beta_n(q)\}_{n\ge 1}$ 
is given by the closed explicit formula
\begin{equation}\label{eq:beta-det}
\beta_n(q)=q\,\det D_{n-1}(q),
\end{equation}
Equivalently, if $C_n(q)$ is the matrix given by \eqref{eq:MatriX}, then for all $n\geq 0$, we have 
\begin{equation}\label{eq:beta-det-diff}
\beta_n(q)=\frac{\det C_{n+1}(q)-\det C_n(q)}{q^{n+1}}.
\end{equation}
\end{proposition}

\begin{proof}
Set, for $m\ge0$, $K_m(q):=\det D_m(q)$ with $K_0(q):=1$. Expanding $\det D_m(q)$ along the last row (continuant recurrence) yields, for $m\ge2$,
\[
K_m(q)=d_{m+1}(q)\,K_{m-1}(q)+c_{m+1}(q)\,K_{m-2}(q),
\]
and $K_1(q)=d_2(q)$. Multiplying by $q$ and reindexing via $n=m+1$ gives that
the sequence $qK_{n-1}$ satisfies \eqref{eq:beta-rec} and has
\[
qK_0=q=\beta_1,\qquad qK_1=qd_2=d_2\beta_1+c_2\beta_0=\beta_2.
\]
By uniqueness of solutions to \eqref{eq:beta-rec} with given initial values, we conclude
$\beta_n=qK_{n-1}=q\det D_{n-1}$ for all $n\ge1$, proving \eqref{eq:beta-det}.

Finally, \eqref{eq:beta-det-diff} follows immediately from Proposition~\ref{prop:an-determinant},
which gives $\alpha_n=\det C_n$, substituted into \eqref{eq:F1-from-a}.
\end{proof}

\begin{remark}
The sequence $\beta_n(q)/q$ is the continuant associated with the
shifted coefficient sequences $\{d_k(q)\}_{k\ge2}$ and $\{c_k(q)\}_{k\ge3}$.
Equivalently, it is the determinant of the principal submatrix obtained
from $C_{n+1}(q)$ by deleting its first row and first column.
Thus $\beta_n$ appears naturally as a minor of the continuant matrix
that represents $\alpha_n$.
\end{remark}

\section{Continued Fraction Formulation}\label{sec:ContinuerFraction}
The normalized two-state recurrence can also be expressed in linear-fractional form. 
This perspective leads to a continued-fraction representation for finite truncations of the generating function and connects the model to classical convergent theory.

Recall the normalized transfer matrices \eqref{eq:NormlizedMatrix} and   the normalized row-vectors \eqref{eq:NormlizedVector}. The finite truncation of the full generating function is defined by
\begin{equation}\label{eq:FN-def}
F_N(q):=(1,0)\Bigl(\prod_{j=1}^N M_j(q)\Bigr)\binom11
=\frac{1}{(q)_N}\bigl(\widehat F^{(N)}_0+\widehat F^{(N)}_1\bigr),
\end{equation}
with $(q)_N:=\prod_{j=1}^N(1-q^j)$.
\begin{proposition}\label{prop:mobius-cf}
Set $r_n(q):=\dfrac{\widehat F^{(n)}_1(q)}{\widehat F^{(n)}_0(q)}$, with $r_0(q)=0$. 
Then for every $n\ge 1$ we have the   (linear-fractional) recurrence
\begin{equation}\label{eq:rn-mobius}
r_n(q)=G_n\bigl(r_{n-1}(q)\bigr),
\end{equation}
with 
\begin{equation}
    G_n(x):=\frac{q^n+(1-q^n)x}{1+q^n x}.
\end{equation}
Equivalently, $r_n$ admits the explicit finite continued fraction by nesting
\begin{equation}\label{eq:rn-nested}
r_n(q)=
\cfrac{q^n+(1-q^n)\,\cfrac{q^{n-1}+(1-q^{n-1})\,\cfrac{\ddots\ \cfrac{q^2+(1-q^2)\,\cfrac{q}{1}}{1+q^2\,\cfrac{q}{1}}\ \ddots}{1+q^{n-1}(\cdots)}}{1+q^{n-1}(\cdots)}}
{1+q^n\,\cfrac{q^{n-1}+(1-q^{n-1})\,\cfrac{\ddots\ \cfrac{q^2+(1-q^2)\,\cfrac{q}{1}}{1+q^2\,\cfrac{q}{1}}\ \ddots}{1+q^{n-1}(\cdots)}}{1+q^{n-1}(\cdots)}}.
\end{equation}
Moreover,
\begin{equation}\label{eq:FN-in-rn}
F_N(q)=\frac{1}{(q)_N}\,\widehat F^{(N)}_0(q)\,\bigl(1+r_N(q)\bigr).
\end{equation}
\end{proposition}
\begin{proof}
From the normalized two-state recurrence \eqref{eq:FirstRec} and \eqref{eq:SecondRec}, divide the second identity by the first to obtain
\[
r_n=\frac{q^n+(1-q^n)r_{n-1}}{1+q^n r_{n-1}}=G_n(r_{n-1}),
\]
which is \eqref{eq:rn-mobius}. Iterating from $r_0=0$ yields the nested expression
\eqref{eq:rn-nested}. Finally, \eqref{eq:FN-in-rn} follows from
\eqref{eq:FN-def} by factoring $\widehat F^{(N)}_0$ from
$\widehat F^{(N)}_0+\widehat F^{(N)}_1=\widehat F^{(N)}_0(1+r_N)$.
\end{proof}

The iteration \eqref{eq:rn-mobius} shows that $r_n$ is a rational function in $q$ (a finite
convergent). In particular, writing
\[
r_n(q)=\frac{P_n(q)}{Q_n(q)},
\]
the pair $(P_n,Q_n)$ can be recovered by multiplying the $2\times 2$ matrices encoding
the   maps $G_n$, so that $\{P_n/Q_n\}_{n\ge 0}$ forms a natural sequence of convergents
associated with the normalized transfer-matrix product.

Proposition~\ref{prop:mobius-cf} improves the presentation in three complementary ways:
\begin{enumerate}
\item It makes transparent that the normalized transfer-matrix product produces a sequence of
\emph{finite convergents} (rational functions in $q$), namely the iterates $r_n(q)$ and the truncations
$F_N(q)$.
\item It clarifies why a ``closed form'' (in the sense of a simple finite product or a single
$q$-hypergeometric term) is not automatic: after uncoupling, one obtains a scalar recurrence with
\emph{$n$-dependent coefficients} through the powers $q^n$.
\item It nevertheless yields a clean analytic object: $r_n$ is obtained by \emph{iterated  
transformations} (equivalently, a nested finite continued fraction), providing a standard
continued-fraction/convergents framework for further analysis.
\end{enumerate}

\begin{corollary}\label{cor:FN-delta}
Let $F_N(q)$ be defined by \eqref{eq:FN-def}. For all $N\geq 1$, define
\begin{equation}
    \Delta_N(q):=(q)_N F_N(q)-(q)_{N-1}F_{N-1}(q),
\end{equation}
with $(q)_0F_0=1$ and $\Delta_1(q)=q$.
Then for all $N\ge 2$ the sequence $\Delta_N(q)$ satisfies the explicit second-order recurrence
\begin{equation}\label{eq:delta-second}
\Delta_{N+1}(q)
= q\bigl(1+q-q^{N}\bigr)\,\Delta_{N}(q)
+ q^2\bigl(q^{2N-1}+q^{N}-q\bigr)\,\Delta_{N-1}(q).
\end{equation}
Moreover, the convergents are recovered by
\begin{equation}\label{eq:FN-from-delta}
(q)_N F_N(q)=1+\sum_{j=1}^N \Delta_j(q),
\qquad\text{equivalently}\qquad
F_N(q)=\frac{1+\sum_{j=1}^N \Delta_j(q)}{(q)_N}.
\end{equation}
\end{corollary}

\begin{proof}
Write $(\widehat F^{(N)}_0,\widehat F^{(N)}_1)$ for the normalized state-sums, so that
$(q)_N F_N=\widehat F^{(N)}_0+\widehat F^{(N)}_1$ by \eqref{eq:FN-def}.
From the first normalized recurrence
$\widehat F^{(N)}_0=\widehat F^{(N-1)}_0+q^N\widehat F^{(N-1)}_1$ we obtain
\[
(q)_N F_N-(q)_{N-1}F_{N-1}
=\bigl(\widehat F^{(N)}_0+\widehat F^{(N)}_1\bigr)
-\bigl(\widehat F^{(N-1)}_0+\widehat F^{(N-1)}_1\bigr)
=q^N\widehat F^{(N-1)}_0,
\]
hence $\Delta_N=q^N\widehat F^{(N-1)}_0$.
Therefore $\Delta_{N+1}/q^{N+1}=\widehat F^{(N)}_0$ satisfies the uncoupled second-order
recurrence for $\widehat F^{(N)}_0$ (Proposition~\ref{prop:uncouple}); multiplying that
recurrence by $q^{N+1}$ yields \eqref{eq:delta-second}.
Finally, \eqref{eq:FN-from-delta} is immediate by telescoping:
$(q)_N F_N=(q)_0F_0+\sum_{j=1}^N\bigl((q)_jF_j-(q)_{j-1}F_{j-1}\bigr)$.
\end{proof}

\begin{remark}
The quantity $\Delta_N(q)=q^N\widehat F^{(N-1)}_0(q)$ has a direct combinatorial interpretation:
it is precisely the total generating function weight of all block-separated overpartitions
whose largest part is $N$, obtained by taking any admissible configuration up to size $N-1$
ending in state $0$ (i.e., whose last present block is plain) and adjoining a plain block
of size $N$.
\end{remark}

\section{Growth and Asymptotic Behavior}\label{sec:Asymptotics}

We conclude the structural analysis by determining the asymptotic growth of the counting function $b(n)$. 
Using the Euler factorization and convolution methods for $q$-series, we show that block-separated overpartitions share the same exponential growth rate as ordinary partitions.  Our analysis is based on the factorization
\begin{equation}\label{eq:factorization}
F(q)=\sum_{n\ge0} b(n)q^n
=
\frac{H(q)}{(q)_\infty},
\end{equation}
where $(q)_\infty=\prod_{j\ge1}(1-q^j)$ and
\begin{equation}
    H(q)=\widehat F^{(\infty)}_0(q)+\widehat F^{(\infty)}_1(q).
\end{equation}
\begin{proposition}[Analytic structure of $H(q)$]\label{prop:Hanalytic}
The function $H(q)$ is analytic for $|q|<1$ and admits a finite limit
\begin{equation}
    H(1):=\lim_{q\to1^-}H(q)\in(0,\infty).
\end{equation}
\end{proposition}

\begin{proof}
Each finite approximation $\widehat F^{(n)}_0(q)$ and
$\widehat F^{(n)}_1(q)$ is a polynomial in $q$.
For $|q|<1$, the infinite product defining $H(q)$ converges coefficientwise,
hence $H(q)$ is analytic in the open unit disc.

Moreover, for $q\in(0,1)$, all coefficients are nonnegative and
$F(q)$ is finite.
Since $(q)_\infty^{-1}\to\infty$ as $q\to1^-$,
Equation~\eqref{eq:factorization} implies that $H(q)$
must remain bounded near $q=1$.
Monotonicity in $q$ then ensures the existence of a finite limit $H(1)>0$.
\end{proof}

It is classical (Hardy--Ramanujan \cite{Andrews_1984}) that the partition function satisfies
\begin{equation}
    p(n)\sim
\frac{1}{4\sqrt3\,n}
\exp\!\left(\pi\sqrt{\frac{2n}{3}}\right).
\end{equation}
This is of the subexponential form
\[
p(n)\sim
v\, n^{-1}\exp\!\left(r n^{1/2}\right),
\qquad
v=\frac{1}{4\sqrt3},
\quad
r=\pi\sqrt{\frac{2}{3}}.
\]

Since $H(q)$ is analytic at $q=1$,
its coefficients $h(n)$ grow at most polynomially.
Equation~\eqref{eq:factorization} implies
\[
b(n)=\sum_{k=0}^n h(k)\,p(n-k),
\]
i.e., $b(n)$ is the convolution of $p(n)$ with a polynomially growing sequence.

By Theorem~1 of Kotěšovec~\cite{kotesovec2015method}
(applied with $p=\tfrac12$),
the convolution of a subexponential sequence
$v n^{-1} e^{r\sqrt n}$ with a polynomially bounded sequence
preserves the same exponential rate and modifies only the leading constant.
Therefore, we have the following main asymptotic formula.

\begin{theorem}
As $n\to\infty$,
\[
b(n)\sim H(1)\,p(n).
\]
Equivalently,
\[
b(n)\sim
\frac{H(1)}{4\sqrt3\,n}
\exp\!\left(\pi\sqrt{\frac{2n}{3}}\right).
\]
\end{theorem}

\begin{proof}
Write $b(n)=\sum_{k=0}^n h(k)p(n-k)$.
Since $h(k)$ grows at most polynomially and
$p(n-k)$ has subexponential form
$v(n-k)^{-1}\exp(r\sqrt{n-k})$,
Theorem~1 of~\cite{kotesovec2015method}
implies that the dominant contribution arises from small $k$,
and that
\[
b(n)\sim
\left(\sum_{k\ge0}h(k)\right)p(n).
\]
But
\[
\sum_{k\ge0}h(k)=H(1),
\]
which completes the proof.
\end{proof}

\begin{corollary}
We have
\[
\log b(n)\sim
\pi\sqrt{\frac{2n}{3}}.
\]
In particular, block-separated overpartitions have the same exponential
growth scale as ordinary partitions.
\end{corollary}

Thus the Fibonacci-type decoration mechanism affects only the
subexponential structure of the counting function, while the
dominant exponential behaviour remains governed by the classical
Euler product.

\section{Conclusion and Further Directions}\label{sec:conclusion}

We introduced block-separated overpartitions as a natural intermediate family between classical partitions and unrestricted overpartitions. 
Although the defining restriction is purely local — forbidding two consecutive distinct blocks from being simultaneously overlined — it generates a remarkably rich global structure. 
The resulting model combines an Euler-type multiplicative structure with a Fibonacci-type internal decoration mechanism, producing a hybrid object that is both combinatorially   and analytically structured.

The model naturally suggests several directions for further investigation. 
One may consider $k$-separated variants in which runs of $k$ consecutive overlined blocks are forbidden, leading to higher-order Fibonacci-type recurrences and $k$-step automata. 
Refined enumerations tracking the number of overlined blocks or the number of distinct part-sizes could produce multivariate generating functions with richer analytic behavior. 
From a $q$-series perspective, the factorization $F(q)=H(q)/(q)_\infty$ raises the question of whether $H(q)$ admits modular-type properties or can be embedded into known $q$-hypergeometric frameworks. 
Finally, the independent-set interpretation points toward possible bijective connections with lattice-path models, tilings, or other restricted combinatorial structures. Block-separated overpartitions therefore provide a simple yet structurally rich refinement of overpartitions.

%  \section*{Statements and Declarations}

% \subsection*{Funding}
% The author  declares that no funds, grants, or other support were received during the preparation of this manuscript.

% \subsection*{Competing Interests}
% The author declares that they have no conflict of interest.

% \subsection*{Author Contributions}

% The single author contributed to the entirety of the paper.

% \subsection*{Data Availability}
% No data were used in this study.

 \bibliographystyle{plainurl} 

\bibliography{bibou}

\end{document}